\documentclass[a4paper]{amsart}
\usepackage{mathtools}

\newtheorem{definition}{Definition}
\newtheorem{lemma}{Lemma}
\newtheorem{theorem}{Theorem}

\newtheorem*{assumption*}{Assmuption}

\newcommand{\hot}{\ensuremath{{h.o.t.}}}

\title[Equisingularity and isotopy]{A short proof that equisingular branches are isotopic}
\author{P. Fortuny Ayuso}
\date{March 6, 2017}
\begin{document}
\begin{abstract}
  We present a short proof of the fact that two irreducible germs of plane analytic curves are isotopic if they are equisingular, without recourse to the structure of the associated knots.
\end{abstract}
\maketitle
The classical proof given by Brauner \cite{Brauner}, which can be made simpler with arguments by K{\"a}hler \cite{Kahler} (see, for example \cite{Wall} for a modern approach) that two irreducible equisingular curves are topologically equivalent uses the structure of the link associated to the singularity and requires a somewhat long reasoning. We expect the following proof simplifies the argument and makes the result easier to understand. Unfortunately, we do not know (yet?) how to simplify the proof of the reciprocal statement.

The notion of equisingularity we use is the following:
\begin{definition}
  Two germs of analytic branches in $(\mathbb{C}^2,0)$ are equisingular if the dual graphs of their desingularisations are equal.
\end{definition}

We shall prove:
\begin{theorem}\label{the:equisingular-implies-homeomorphic}
  If two irreducible germs of analytic curves are equisingular then there exists an isotopy $\left\{ \psi_t \right\}_{t\in[0,1]}$ of $(\mathbb{C}^2,0)$ with $\psi_0=\mathrm{Id}$, $\psi_1(\Gamma_1) = \Gamma_2$.
\end{theorem}
Before proceeding with the actual proof, we need two elementary results. Consider $(\mathbb{C}^2,0)$ with coordinates $(x,y)$.
\begin{lemma}\label{lem:non-singular-equivalent}
  Given two regular germs of analytic curves $\Gamma_1, \Gamma_2$ at $(\mathbb{C}^2,0)$ both transverse to both $(y=0)$ and $(x=0)$, there is an analytic vector field $X$ in $(\mathbb{C}^2,0)$ such that the associated isotopy $\left\{ \phi_t \right\}_{t\in[0,1]}$ sends $\Gamma_1$ to $\Gamma_2$ and keeps both $(y=0)$ and $(x=0)$ invariant for all $t\in[0,1]$. 
\end{lemma}
\begin{proof}
  After an analytic change of coordinates, we may assume without loss of generality, that $\Gamma_1\equiv y=x$ and $\Gamma_2 \equiv (y=s(x)=cx + \hot)$ with $c\neq 0$. Consider an open neighbourhood $U$ of $(0,0)$ and in $U_1\subset \overline{U_1}\subset U$, the vector field
  \begin{equation*}
    X = \left(0, \log\left(\frac{s(x)}{x}\right)y\right),
  \end{equation*}
  (for some determination of the logarithm), and the zero vector field on the complement of a ball included in $U_1$. The glued vector field by means of a partition of unity gives the result.
\end{proof}

The second lemma deals with modifying the tangent cone of an irreducible singularity.

\begin{lemma}\label{lem:move-tangent}
  Let $\Gamma$ be a (singular or not) germ of analytic curve at $(\mathbb{C}^2,0)$ and let $(x,y)$ be local coordinates such that $\Gamma$ is not tangent to either $(x=0)$ or $(y=0)$. Given any line $L\equiv y=ax$ with $a\neq 0$, there is a holomorphic vector field in $(\mathbb{C}^2,0)$ such that the associated isotopy $\left\{ \phi_t \right\}_{t\in[0,1]}$ makes $\phi_1(\Gamma)$ tangent to $L$ and keeps both $(x=0)$ and $(y=0)$ invariant for all $t\in[0,1]$.
\end{lemma}
\begin{proof}
  This is an easy consequence of Lemma \ref{lem:non-singular-equivalent} taking $\Gamma_1 =L$ and $\Gamma_2$ the line tangent to $\Gamma$.
\end{proof}

We can now proceed to prove Theorem \ref{the:equisingular-implies-homeomorphic}:

\begin{proof}[Proof of Theorem \ref{the:equisingular-implies-homeomorphic}]
  Let $\Gamma_1$ and $\Gamma_2$ be such germs. Let $r$ be the length of the desingularisation $\pi:\mathcal{X}\rightarrow (\mathbb{C}^2,0) $ of $\Gamma_1$. We proceed by induction on $r-k$, where $k$ is the cardinality of the set of centres in $\pi$ which belong to $\Gamma_2$.

  If $r-k=0$ then $\Gamma_1$ and $\Gamma_2$ have the same desingularisation. Let $P=\overline{\Gamma_1}\cap \overline{\Gamma_2}$ be the intersection of the strict transforms of $\Gamma_1$ and $\Gamma_2$, which is a point where both strict transforms are non-singular and cross the exceptional divisor at a corner, both in directions transverse to both components of the exceptional divisor. Let $X$ be the vector field near $P$ given by Lemma \ref{lem:non-singular-equivalent}. Glue $X$ with the zero vector field in $\mathcal{X}$ away from a sufficiently small neighbourhood of $P$. If $\phi_t$ is the associated diffeomorphism for $t\in[0,1]$, then the family $\left\{ \pi \circ \phi_t \right\}_{t\in[0,1]}$ gives the desired isotopy of $(\mathbb{C}^2,0)$ between the branches.

  Assume the result true for $r-k=i$ and let $r-k=i+1$. Let $Q$ be the last common centre in the desingularisations of $\Gamma_1$ and $\Gamma_2$. As the curves are equisingular, the only possibility is that their strict transforms $\overline{\Gamma_1}$ and $\overline{\Gamma_2}$ at $Q$ are neither of them tangent to any component of the exceptional divisor at $Q$ and that their tangent cones at $Q$ are different. By means of Lemma \ref{lem:move-tangent}, construct a vector field near $Q$ whose associated isotopy makes $\phi_1(\overline{\Gamma_2})$ tangent to $\overline{\Gamma_1}$. Gluing this with the zero vector field away from a sufficiently small neighbourhood of $Q$ and projecting to $(\mathbb{C}^2,0)$, we get an isotopy between $\Gamma_2$ and a branch for which $r-k\leq i$, which finishes the argument.
\end{proof}


\end{document}